\documentclass[12pt]{amsart}

\usepackage{amsmath,amssymb,amsfonts,amsthm,latexsym,graphicx,multirow}
\usepackage{hyperref}

     \def\Cay{\mathrm{Cay}}  
\def\D{\mathrm{D}}

 \def\Ga{\Gamma} \def\Gal{\mathrm{Gal}}   \def\GF{\mathrm{GF}}

\def\SL{\mathrm{SL}}

\newtheorem{theorem}{Theorem}[section]
\newtheorem{lemma}[theorem]{Lemma}

\newtheorem{corollary}[theorem]{Corollary}

\theoremstyle{definition}

\newtheorem{definition}[theorem]{Definition}

\usepackage{color}

\definecolor{VeryLightBlue}{rgb}{0.9,0.9,1}
\definecolor{LightBlue}{rgb}{0.8,0.8,1}
\definecolor{MidBlue}{rgb}{0.5,0.5,1}
\definecolor{DarkBlue}{rgb}{0,0,0.6}
\definecolor{Blue}{rgb}{0,0,1}

\definecolor{Gold}{rgb}{1,0.843,0}

\definecolor{LightGreen}{rgb}{0.88,1,0.88}
\definecolor{MidGreen}{rgb}{0.6,1,0.6}
\definecolor{DarkGreen}{rgb}{0,0.6,0}

\definecolor{VeryLightYellow}{rgb}{1,1,0.9}
\definecolor{LightYellow}{rgb}{1,1,0.6}
\definecolor{MidYellow}{rgb}{1,1,0.5}
\definecolor{DarkYellow}{rgb}{1,1,0.2}
\definecolor{DarkPurple}{rgb}{.6,0,1}
\definecolor{Red}{rgb}{1,0,0}
\definecolor{VeryLightRed}{rgb}{1,0.9,0.9}
\definecolor{LightRed}{rgb}{1,0.8,0.8}
\definecolor{MidRed}{rgb}{1,0.55,0.55}

\def\b0{{\bf 0}}

\def\ZZZ{\mathbb{Z}}

\long\def\delete#1{}

\usepackage{xcolor}
\usepackage[normalem]{ulem}

\begin{document}

\title[Perfect codes in Cayley graphs]{Perfect codes in Cayley graphs}

\author[Huang]{He Huang}
\address{(Huang) School of Mathematical Sciences\\Peking University\\Beijing 100871, P. R. China}
\email{1301110019@math.pku.edu.cn}

\author[Xia]{Binzhou Xia}
\address{(Xia) School of Mathematics and Statistics\\The University of Melbourne\\Parkville, VIC 3010, Australia}
\email{binzhou.xia@unimelb.edu.au}

\author[Zhou]{Sanming Zhou}
\address{(Zhou) School of Mathematics and Statistics\\The University of Melbourne\\Parkville, VIC 3010, Australia}
\email{sanming@unimelb.edu.au}


\openup 0.5\jot

\maketitle

\begin{center}
\emph{\small Dedicated to Professor Yixun Lin on the occasion of his 80th birthday}
\end{center}

\begin{abstract}
Given a graph $\Ga$, a subset $C$ of $V(\Ga)$ is called a perfect code in $\Ga$ if every vertex of $\Ga$ is at distance no more than one to exactly one vertex in $C$, and a subset $C$ of $V(\Ga)$ is called a total perfect code in $\Ga$ if every vertex of $\Ga$ is adjacent to exactly one vertex in $C$. In this paper we study perfect codes and total perfect codes in Cayley graphs, with a focus on the following themes: when a subgroup of a given group is a (total) perfect code in a Cayley graph of the group; and how to construct new (total) perfect codes in a Cayley graph from known ones using automorphisms of the underlying group. We prove several results around these questions.

\textit{Key words:} perfect code; efficient dominating set; Cayley graph; tilings of groups

\textit{AMS Subject Classification (2010):} 05C25, 05C69, 94B25
\end{abstract}

\section{Introduction}
\label{sec:int}

In the paper all groups considered are finite, and all graphs considered are finite and simple. Let $\Ga=(V(\Ga),E(\Ga))$ be a graph and $t \ge 1$ an integer. A subset $C$ of $V(\Ga)$ is called a \emph{perfect $t$-code} \cite{K86} in $\Ga$ if every vertex of $\Ga$ is at distance no more than $t$ to exactly one vertex of $C$ (in particular, $C$ is an independent set of $\Ga$). A perfect 1-code is simply called a \emph{perfect code}. A subset $C$ of $V(\Ga)$ is said to be a \emph{total perfect code} \cite{Zhou2016} in $\Ga$ if every vertex of $\Ga$ has exactly one neighbour in $C$ (in particular, $C$ induces a matching in $\Ga$ and so $|C|$ is even). In graph theory, a perfect code in a graph is also called an \emph{efficient dominating set} \cite{DS03} or \emph{independent perfect dominating set} \cite{L01}, and a total perfect code is called an \emph{efficient open dominating set} \cite{HHS}.

The concepts above were developed from the work in \cite{Biggs1}, which in turn has a root in coding theory. In the classical setting, a $q$-ary code is a subset of $S^n$, where $q$ and $n$ are positive integers and $S^n$ is the set of words of length $n$ over a set $S$ of size $q$. A $q$-ary code $C \subseteq S^n$ is called a \emph{perfect $t$-code} \cite{Heden1, vL} if every word in $S^n$ is at distance no more than $t$ to exactly one codeword of $C$, where the (Hamming) distance between two words is the number of positions in which they differ. Thus, the $q$-ary perfect $t$-codes of length $n$ in the classical setting are precisely the perfect $t$-codes in Hamming graph $H(n, q)$, the graph with vertex set $S^n$ and edges joining pairs of words of Hamming distance one. Beginning with \cite{Biggs1} and \cite{Del}, perfect codes in distance-transitive graphs and association schemes in general have received considerable attention in the literature. See, for example, \cite{B, Biggs1, HS, Smith1, SE}. Since Hamming graphs are distance-transitive, perfect codes in distance-transitive graphs can be thought as generalizations of perfect codes in the classical setting.

Perfect codes in Cayley graphs are another generalization of perfect codes in the classical setting because $H(n, q)$ is the Cayley graph of $(\ZZZ/q\ZZZ)^n$ with respect to the set of elements of $(\ZZZ/q\ZZZ)^n$ with exactly one nonzero coordinate. We denote by $e$ the identity element of a group $G$, and set $T^{-1}=\{t^{-1}\mid t\in T\}$ for any subset $T$ of $G$. Given a group $G$ and a subset $S$ of $G$ with $e\notin S$ and $S^{-1}=S$, the \emph{Cayley graph} $\Cay(G,S)$ of $G$ with respect to the \emph{connection set} $S$ is defined to be the graph with vertex set $G$ such that $x,y \in G$ are adjacent if and only if $yx^{-1}\in S$. A perfect $t$-code in a Cayley graph $\Cay(G,S)$ (just like in a general graph) can be thought as a tiling of $G$ by the balls of radius $t$ centred at the vertices of $C$ with respect to the graph distance. Moreover, if $X$ is a group of order $q$, $G = X \times \cdots \times X$ is the direct product of $n$ copies of  $X$, and $S$ consists of those elements of $G$ with exactly one non-identity coordinate, then $\Cay(G,S)$ is isomorphic to $H(n,q)$, and every subgroup of $G$ is a group code. So the notion of perfect codes in Cayley graphs can be also viewed as a generalization of the concept of perfect group codes \cite{vL}. Thus the study of perfect codes in Cayley graphs is highly related to both coding theory and group theory.

Perfect codes in Cayley graphs have received considerable attention in recent years.
In \cite{MBG07} sufficient conditions for two families of Cayley graphs of quotient rings of $\ZZZ[i]$ and $\ZZZ[\rho]$ to contain perfect $t$-codes were given, where $i = \sqrt{-1}$ and $\rho = (1+\sqrt{-3})/2$, and these conditions were proved to be necessary in \cite{Z15} in a more general setting. Related works can be found in \cite{MBCSG10} and \cite{MBG09}. In \cite{T04} it was proved that there is no perfect code in any Cayley graph of $\SL(2, 2^f)$, $f > 1$ with respect to a conjugation-closed connection set. In \cite{E87} necessary conditions for the existence of perfect codes in a Cayley graph with a conjugation-closed connection set were obtained by way of irreducible characters of the underlying group. In \cite{L01} connections between perfect codes that are closed under conjugation and a covering projection from the Cayley graph involved to a complete graph were given, and counterpart results for total perfect codes were obtained in \cite{Zhou2016}. Perfect codes in circulants were studied in \cite{YPD14, DSLW16, FHZ16, OPR07, KM13, TM2013} (a circulant is a Cayley graph on a cyclic group). In \cite{DS03}, E-chains of Cayley graphs on symmetric groups were constructed, where an E-chain is a countable family of nested graphs each containing a perfect code.

Perfect codes in Cayley graphs of groups are closely related to factorizations and tilings of groups. A \emph{factorization}\footnote{A related but different notion in group theory not to be studied in this paper is as follows. A factorization of a group $G$ is a pair of subgroups $A, B$ of $G$ such that $G = AB$. In this definition $A$ and $B$ are required to be subgroups of $G$ but each element of $G$ is not required to be expressed uniquely as $ab$ with $a \in A$ and $b \in B$.} \cite{SS} of a group $G$ is a tuple $(A_1, \ldots, A_k)$ of subsets of $G$ such that every element of $G$ can be uniquely represented as $a_1 \cdots a_k$ with each $a_i \in A_i$, so that $G = A_1 \cdots A_k$. Such a factorization is \emph{normed} if $e$ is contained in each factor $A_i$. A \emph{tiling} $(A, B)$
of $G$ is simply a normed factorization $G = AB$ of $G$ into two factors \cite{Dinitz06}. It is readily seen (Lemma \ref{thm1}(a)) that if $(A, B)$ is a tiling of $G$ with $A^{-1} = A$, then $B$ is a perfect code of $\Cay(G, A \setminus \{e\})$. The study of factorizations of abelian groups was initiated by G. Haj\'{o}s \cite{Hajos} in his proof of a well-known conjecture of Minkowski. Since then there has been extensive research on factorizations and tilings of abelian groups; see for example \cite{Dinitz06, Szabo06, SS} for related results and background information. As far as we know, in the literature most studies on tilings are about the abelian case. In the present paper we do not restrict ourselves to abelian groups.

Despite the results above, perfect codes in Cayley graphs deserve further study due to their strong connections with coding theory, group theory and graph theory. In this paper we study perfect codes and total perfect codes in Cayley graphs with a focus on the following themes: (i) when a subgroup of a given group is a (total) perfect code in a Cayley graph of the group; (ii) how to construct new (total) perfect codes in a Cayley graph from known ones using automorphisms of the underlying group. In classical coding theory, linear codes \cite{Heden1, vL} (subspaces of linear spaces over finite fields) have been studied extensively. In some sense, subgroups of a group that are perfect codes in some Cayley graphs of the group are an analog of perfect linear codes. (In fact, when $q$ is a prime power, $H(n,q)$ is a Cayley graph of the additive group of $\GF(q)^n$ and a perfect linear code can be viewed as a subgroup of the additive group of $\GF(q)^n$.) This observation motivated our study of theme (i). In a sense, this theme is dual to the problem of constructing perfect codes (or proving there is none) in a given graph. Theme (i) can also be phrased as the existence of an inverse-closed tiling mate of a given subgroup, which is of interest from a tiling point of view. The study of theme (ii) was motivated by our lack of general methods for constructing (total) perfect codes in Cayley graphs, and as such constructions of new (total) perfect codes from known ones would be desirable.

The main results in this paper are as follows. In section \ref{sec:subgroup} we first give a necessary and sufficient condition for a normal subgroup of a group to be a (total) perfect code in some Cayley graph of the group (Theorem \ref{thm3}). As a corollary we obtain that every normal subgroup with odd order or odd index is a perfect code in some Cayley graph of the group, and every normal subgroup with even order but odd index is a total perfect code in some Cayley graph of the group (Corollary \ref{cor2}). We prove further that this sufficient condition is also necessary when the group is cyclic (Corollary \ref{cor3}). In general, for any abelian group $G$, we obtain a necessary and sufficient condition (Theorem \ref{Abelian}) for a subgroup $H$ of $G$ with $H \cap P$ cyclic to be a (total) perfect code in some Cayley graph of $G$, where $P$ is the Sylow 2-subgroup of $G$. We also obtain a necessary and sufficient condition (Theorem \ref{thm:dih}) for a proper subgroup of the dihedral group $\D_{2n}$ to be a (total) perfect code in some Cayley graph of $\D_{2n}$.

An automorphism of a group $G$ is said to be perfect-code-preserving if it maps every perfect code in any Cayley graph of $G$ to a perfect code in the same Cayley graph. Total-perfect-code-preserving automorphisms are understood similarly. In section \ref{sec:pcp} we first prove that all power automorphisms of any abelian group $G$ are both perfect-code-preserving and total-perfect-code-preserving, and in particular so is the automorphism of $G$ that maps $g \in G$ to $g^n$ for any fixed integer $n \ge 1$ coprime to the order of $G$ (Corollary \ref{thm4}). We then prove that every perfect-code-preserving inner automorphism of any group must be a power automorphism (Theorem \ref{prop3}).

\section{Subgroup perfect codes}
\label{sec:subgroup}

For brevity, if a subset $C$ of a group $G$ is a perfect code (respectively, total perfect code) in some Cayley graph of $G$, then $C$ is simply called a \emph{perfect code} (respectively, \emph{total perfect code}) of $G$. In this section we study when a subgroup of a given group is a perfect code or total perfect code of the group. Such a subgroup is called a \emph{subgroup perfect code} or \emph{subgroup total perfect code} of the group.

The following results follow directly from the definitions of perfect and total perfect codes.

\begin{lemma}\label{Transversal}
Let $\Ga=\Cay(G,S)$ be a Cayley graph of a group $G$, and let $H$ be a subgroup of $G$.
\begin{itemize}
\item[(a)] $H$ is a perfect code in $\Ga$ if and only if $S\cup\{e\}$ is a left transversal of $H$ in $G$.
\item[(b)] $H$ is a total perfect code in $\Ga$ if and only if $S$ is a left transversal of $H$ in $G$.
\end{itemize}
\end{lemma}

\subsection{Normal subgroups}

We now establish a necessary and sufficient condition for a normal subgroup $H$ of a group $G$ to be a perfect code or total perfect code of $G$ in terms of the following property: for any $g\in G$, $g^2\in H$ implies $(gh)^2=e$ for some $h\in H$, that is,
\begin{equation}\label{Property}
\forall g\in G\ (g^2\in H)\ \exists h\in H\ ((gh)^2=e).
\end{equation}

\begin{theorem}\label{thm3}
Let $G$ be a group and $H$ a normal subgroup of $G$.
\begin{itemize}
\item[(a)] $H$ is a perfect code of $G$ if and only if~\eqref{Property} holds.
\item[(b)] $H$ is a total perfect code of $G$ if and only if $|H|$ is even and~\eqref{Property} holds.
\end{itemize}
\end{theorem}

\begin{proof}
Suppose that $H$ is a perfect code or a total perfect code in a Cayley graph of $G$. Then by Lemma~\ref{Transversal} there exists a left transversal $L$ of $H$ in $G$ such that $L^{-1}=L$. Let $g$ be an element of $G$ with $g^2\in H$. Since $L$ is a left transveral, there exists an element $g_1$ of $L$ such that $g\in g_1H$. Then $gH=g_1H$ and $g^{-1}H=Hg^{-1}=Hg_1^{-1}=g_1^{-1}H$ as $H$ is normal in $G$. However, $g^2\in H$ implies that $g^{-1}H=gH$. Hence $g_1^{-1}H=g_1H$, which implies $g_1^{-1}=g_1$ since $g_1 \in L$ and $g_1^{-1}\in L^{-1}=L$. Now $g_1^2=e$, and from $g\in g_1H$ we obtain that $g_1=gh$ for some $h\in H$. It follows that $(gh)^2=e$. This together with the fact that the size of every total perfect code must be even proves the ``only if" part in statements (a) and (b).

Suppose that $H$ is a normal subgroup of $G$ such that for any $g\in G$, $g^2\in H$ implies
$(gh)^2=e$ for some $h\in H$. Let ${\mathcal I} = \{x_1H,\dots,x_sH\}$ be the set of involutions in $G/H$, and let ${\mathcal J}$ be the set of elements of order at least $3$ in $G/H$. Then there exist $g_1,\dots,g_t\in G$ such that
\[
{\mathcal J}=\{g_1H,(g_1H)^{-1},\dots,g_tH,(g_tH)^{-1}\}=\{g_1H,g_1^{-1}H,\dots,g_tH,g_t^{-1}H\}.
\]
For any element $x_iH$ of ${\mathcal I}$, we have $(x_iH)^2=H$, that is, $x_i^2\in H$. Hence there exists $h_i\in H$ such that $(x_ih_i)^2=e$. Denote $y_i=x_ih_i$ for $i\in\{1,\dots,s\}$, and take $S=\{y_1,\dots,y_s,g_1,g_1^{-1},\dots,g_t,g_t^{-1}\}$. Then $S\cup\{e\}$ is a left transversal of $H$ in $G$, and $S^{-1}=S$. In addition, if $|H|$ is even, then $H$ contains an involution $g_0$ and $R:=S\cup\{g_0\}$ is a left transversal of $H$ in $G$ with $R^{-1}=R$. By Lemma~\ref{Transversal}, $H$ is a perfect code in $\Cay(G,S)$, and $H$ is a total perfect code in $\Cay(G,R)$ if $|H|$ is even. This completes the proof of the ``if" part in statements (a) and (b).
\end{proof}

Note that the proof above gives a construction of a Cayley graph that admits a given normal subgroup satisfying \eqref{Property} as a perfect code, as well as a construction of a Cayley graph that admits a given normal subgroup of even order satisfying \eqref{Property} as a total perfect code.

\begin{corollary}\label{cor2}
Let $G$ be a group and $H$ a normal subgroup of $G$.
\begin{itemize}
\item[(a)] If either $|H|$ or $|G/H|$ is odd, then $H$ is a perfect code of $G$.
\item[(b)] If $|H|$ is even and $|G/H|$ is odd, then $H$ is a total perfect code of $G$.
\end{itemize}
\end{corollary}

\begin{proof}
Let $g$ be an element of $G$ such that $g^2\in H$. First assume that $|H|$ is odd. Then $(g^2)^{2k+1}=e$ for some nonnegative integer $k$. Taking $h=g^{2k}$, we have $h=(g^2)^k\in H$ and $(gh)^2=(g\cdot g^{2k})^2=(g^2)^{2k+1}=e$. Next assume that $|G/H|$ is odd. Then $gH$ is the identity element of $G/H$ if and only if $(gH)^2$ is. Hence $g^2\in H$ implies $g\in H$. Taking $h=g^{-1}$, we have $h\in H$ and $(gh)^2=e^2=e$.

We have proved that, under the assumption of statement (a) or (b), for any $g\in G$, $g^2\in H$ implies $(gh)^2=e$ for some $h\in H$. The corollary then follows from Theorem~\ref{thm3}.
\end{proof}

The following is an immediate consequence of Corollary~\ref{cor2}.

\begin{corollary}\label{Odd}
Let $G$ be a group of odd order.
\begin{itemize}
\item[(a)] Any normal subgroup of $G$ is a perfect code of $G$.
\item[(b)] $G$ has no subgroup total perfect code.
\end{itemize}
\end{corollary}

\subsection{Abelian groups}

We now use Theorem~\ref{thm3} to study subgroup perfect codes and subgroup total perfect codes of abelian groups.

\begin{lemma}\label{lem30}
Let $G$ be an abelian group with Sylow $2$-subgroup $P$, and let $H$ be a subgroup of $G$. Then $H$ is a (total) perfect code of $G$ if and only if $H\cap P$ is a (total) perfect code of $P$.
\end{lemma}

\begin{proof}
Since $|H|$ is even if and only if $|H\cap P|$ is even, by Theorem~\ref{thm3}, we only need to prove that~\eqref{Property} holds if and only if the following holds:
\begin{equation}\label{eq30}
\forall g\in P\ (g^2\in H\cap P)\ \exists h\in H\cap P\ ((gh)^2=e).
\end{equation}

First suppose that~\eqref{Property} holds. Then for any $g\in P$ with $g^2\in H\cap P$, there exists $h\in H$ such that $(gh)^2=e$. Since $h^2=g^{-2}\in P$, it follows that $h\in P$, and so we have $h\in H\cap P$ with $(gh)^2=e$. This shows that~\eqref{eq30} holds.

Next suppose that~\eqref{eq30} holds. Let $Q$ be the Hall $2'$-subgroup of $G$. Clearly, $G=P\times Q$ and $H=(H\cap P)\times(H\cap Q)$. Now for any $g\in G$ such that $g^2\in H$, write $g=g_1g_2$ with $g_1\in P$ and $g_2\in Q$. Then $g_1^2g_2^2=g^2\in H$ implies that $g_1^2\in H\cap P$ and $g_2^2\in H\cap Q$. Since $g_1^2\in H\cap P$, we derive from~\eqref{Property} that there exists $h_1\in H\cap P$ with $(g_1h_1)^2=e$. Since $g_2^2\in H\cap Q$ and $|H\cap Q|$ is odd, we deduce that $g_2\in H\cap Q$. Hence $h:=h_1g_2^{-1}$ is an element of $H$ such that $(gh)^2=(g_1g_2h_1g_2^{-1})^2=(g_1h_1)^2=e$. This shows that~\eqref{Property} holds.
\end{proof}

\begin{lemma}\label{lem31}
Let $G$ be an abelian group with Sylow $2$-subgroup $P=\langle a_1\rangle\times\dots\times\langle a_n\rangle$, where $a_i$ has order $2^{m_i}>1$ for $1\leq i\leq n$. Let $H$ be a subgroup of $G$.
\begin{itemize}
\item[(a)] If $H$ is a perfect code of $G$, then either $H\cap P$ is trivial or $H\cap P$ projects onto at least one of $\langle a_1\rangle,\dots,\langle a_n\rangle$.
\item[(b)] If $H$ is a total perfect code of $G$, then $H\cap P$ projects onto at least one of $\langle a_1\rangle,\dots,\langle a_n\rangle$.
\end{itemize}
\end{lemma}

\begin{proof}
By virtue of Lemma~\ref{lem30}, we may assume $G=P$. Then $G=\langle a_1\rangle\times\dots\times\langle a_n\rangle$. The following proof is valid for both (a) and (b).

By Theorem~\ref{thm3}, in order to prove (a) and (b) it suffices to prove that if~\eqref{Property} holds and $|H| > 1$, then $H$ projects onto at least one of $\langle a_1\rangle,\dots,\langle a_n\rangle$. Denote by $\pi_i$ the projection of $G$ onto $\langle a_i\rangle$ for $1\leq i\leq n$. Suppose that~\eqref{Property} holds and $|H|>1$. Suppose for a contradiction that $\pi_i(H)\neq\langle a_i\rangle$ for all $1\leq i\leq n$. Since $|H|>1$, at least one of $\pi_1(H),\dots,\pi_n(H)$, say $\pi_1(H)$, is nontrivial. This means that $\pi(H)=\langle a_1^{2^\ell}\rangle$ with $0<\ell<m_1$. Consequently, $a_1^{2^\ell}a_2^{s_2}\dots a_n^{s_n}\in H$ for some integers $s_2,\dots,s_n$. Since $\pi_i(H)\neq\langle a_i\rangle$ for $2\leq i\leq n$, the integers $s_2,\dots,s_n$ are all even. Take $g=a_1^{2^{\ell-1}}a_2^{s_2/2}\cdots a_n^{s_n/2}\in G$. Then $g^2=a_1^{2^\ell}a_2^{s_2}\cdots a_n^{s_n}\in H$, and hence by~\eqref{Property} there exists $h\in H$ such that $h^{-2}=g^2$. However, this implies that $h^{-1}\in H$ and $(\pi_1(h^{-1}))^2=\pi_1(g^2)=a_1^{2^\ell}$, which is impossible as $\pi_1(h^{-1})\in\pi_1(H)=\langle a_1^{2^\ell}\rangle$ and $\ell<m_1$.
\end{proof}

The condition in Lemma \ref{lem31} is not sufficient for $H$ to be a perfect code or total perfect code of $G$. For example, if $G=\langle a_1\rangle\times\langle a_2\rangle\times\langle a_3\rangle$ with $a_1$, $a_2$ and $a_3$ of order $2$, $4$ and $4$, respectively, and $H=\langle a_1a_2^2\rangle\times\langle a_1a_3^2\rangle$, then $(a_2 a_3)^2=a_2^2 a_3^2\in H$ while $a_2^2 a_3^2$ is not a square of any element of $H$. However, the condition in Lemma~\ref{lem31} is sufficient if $H\cap P$ is cyclic, as shown in the next theorem.

\begin{theorem}\label{Abelian}
Let $G$ be an abelian group with Sylow $2$-subgroup $P=\langle a_1\rangle\times\dots\times\langle a_n\rangle$, where $a_i$ has order $2^{m_i}>1$ for $1\leq i\leq n$. Suppose that $H$ is a subgroup of $G$ such that $H\cap P$ is cyclic.
\begin{itemize}
\item[(a)] $H$ is a perfect code of $G$ if and only if either $H\cap P$ is trivial or $H\cap P$ projects onto at least one of $\langle a_1\rangle,\dots,\langle a_n\rangle$.
\item[(b)] $H$ is a total perfect code of $G$ if and only if $H\cap P$ projects onto at least one of $\langle a_1\rangle,\dots,\langle a_n\rangle$.
\end{itemize}
\end{theorem}

\begin{proof}
By Lemma~\ref{lem30}, we may assume $G=P$. Then $G=\langle a_1\rangle\times\dots\times\langle a_n\rangle$ and $H$ is cyclic. By virtue of Lemma~\ref{lem31} and Theorem~\ref{thm3}, we only need to prove that if either $H$ is trivial or $H$ projects onto at least one of $\langle a_1\rangle,\dots,\langle a_n\rangle$ then~\eqref{Property} holds. If $|H|=1$, then~\eqref{Property} holds trivially. Now assume that $H$ projects onto at least one of $\langle a_1\rangle,\dots,\langle a_n\rangle$, say $\langle a_1\rangle$. Then $H=\langle a_1a_2^{s_2}\cdots a_n^{s_n}\rangle$ for some integers $s_2,\dots,s_n$. Denote by $\pi_1$ the projection of $G$ onto $\langle a_1\rangle$. Let $g$ be an element of $G$ such that $g^2\in H$, and write $g^2=(a_1a_2^{s_2}\cdots a_n^{s_n})^t$ for some integer $t$. If $t$ is odd, then we obtain a contradiction that $(\pi_1(g))^2=\pi_1(g^2)=a_1^t$ as the order of $a_1$ is even. Hence $t$ is even. Taking $h=(a_1a_2^{s_2}\cdots a_n^{s_n})^{-t/2}$, we have $(gh)^2=e$. This implies that ~\eqref{Property} holds, completing the proof.
\end{proof}

\begin{corollary}\label{cor3}
Let $G$ be a cyclic group and $H$ a subgroup of $G$.
\begin{itemize}
\item[(a)] $H$ is a perfect code of $G$ if and only if either $|H|$ or $|G/H|$ is odd.
\item[(b)] $H$ is a total perfect code of $G$ if and only if $|H|$ is even and $|G/H|$ is odd.
\end{itemize}
\end{corollary}

\begin{proof}
Let $P$ be the Sylow $2$-subgroup of $G$. By Theorem~\ref{Abelian}, $H$ is a perfect code of $G$ if and only if either $|H\cap P|=1$ or $H\cap P=P$. Note that $|H\cap P|=1$ if and only if $|H|$ is odd, and $H\cap P=P$ if and only if $|G/H|$ is odd. We see that part~(a) is true. Similarly, part~(b) is true.
\end{proof}

\subsection{Dihedral groups}

To understand tilings of non-abelian groups and in particular perfect codes in Cayley graphs of non-abelian groups, it is natural to start with dihedral groups. With this in mind in this subsection we determine the subgroup perfect codes and subgroup total perfect codes of dihedral groups. Let $\D_{2n}=\langle a,b\mid a^n=b^2=(ab)^2=e\rangle$ be the dihedral group of order $2n$, where $n\geq3$. The subgroups of $\D_{2n}$ are the cyclic groups $\langle a^t\rangle$ with $t$ dividing $n$ and the dihedral groups $\langle a^t,a^sb\rangle$ with $t$ dividing $n$ and $0\leq s\leq t-1$.

\begin{lemma}\label{lem32}
Let $\D_{2n} = \langle a,b\mid a^n=b^2=(ab)^2=e\rangle$ and $H=\langle a^t\rangle$ with $t$ dividing $n$.
\begin{itemize}
\item[(a)] $H$ is a perfect code of $\D_{2n}$ if and only if either $t$ or $n/t$ is odd.
\item[(b)] $H$ is a total perfect code of $\D_{2n}$ if and only if $t$ is odd and $n/t$ is even.
\end{itemize}
\end{lemma}

\begin{proof}
Since $H$ is normal in $\D_{2n}$, by Theorem~\ref{thm3} we only need to prove that~\eqref{Property} holds if and only if either $t$ or $n/t$ is odd.

First suppose that~\eqref{Property} holds and both $t$ and $n/t$ are even. Take $g=a^{t/2}$. Then $g\notin H$ and $g^2\in H$. Hence \eqref{Property} asserts that there exists $h\in H$ with $(gh)^2=e$. Since both $g$ and $h$ are in $\langle a\rangle$, we have $gh\in H$. Then the equality $(gh)^2=e$ implies that $gh=a^{n/2}$. However, $n/2$ is divisible by $t$ since $n/t$ is even, and so $a^{n/2}\in H$. It follows that $g=a^{n/2}h^{-1}\in H$, a contradiction. Thus \eqref{Property} implies that either $t$ or $n/t$ is odd.

Next suppose that either $t$ or $n/t$ is odd. Let $g$ be an element of $\D_{2n}$ such that $g^2\in H$. If $g\notin \langle a \rangle$, then $g^2=e$ and taking $h=e\in H$ we have $(gh)^2=e$. Now assume $g\in \langle a \rangle$, which means that $g=a^k$ for some integer $k$. If $t$ is odd, then $a^{2k}\in\langle a^t\rangle$ implies $a^k\in\langle a^t\rangle$ and so $g\in H$. If $n/t$ is odd, then $|H|$ is odd and so $g^2\in H$ implies $g\in H$. Hence we always have $(gh)^2=e$ with $h:=g^{-1}\in H$. This shows that~\eqref{Property} holds.
\end{proof}

In order to prove our next result (Theorem \ref{thm:dih}) and two other results in the next section, we will use the following lemma that gives equivalent definitions of perfect codes and total perfect codes in Cayley graphs using group ring $\mathbb{Z}[G]$ or tilings of $G$. For a group $G$ and a subset $A$ of $G$, denote
$$
\overline{A}=\sum_{g\in G}\mu_A(g)g\in\mathbb{Z}[G],
$$
where
\begin{equation}
\label{eq:mu}
\mu_A(g) = \begin{cases}
1,\; & g\in A; \\
0,\; & g\in G\setminus A.
\end{cases}
\end{equation}

\begin{lemma}\label{thm1}
Let $G$ be a group and $\Ga=\Cay(G,S)$ a Cayley graph of $G$. Let $C$ be a subset of $G$.
\begin{itemize}
\item[(a)] $C$ is a perfect code in $\Ga$ if and only if $\overline{S\cup\{e\}}\cdot\overline{C}=\overline{G}$.
\item[(b)] $C$ is a total perfect code in $\Ga$ if and only if $\overline{S}\cdot\overline{C}=\overline{G}$.
\end{itemize}
\end{lemma}

Note that when $e \in C$ the condition $\overline{S\cup\{e\}}\cdot\overline{C}=\overline{G}$ above is equivalent to saying that $(S \cup \{e\}, C)$ is a tiling of $G$, and that the assumption $e \in C$ does not sacrifice generality since for any perfect code $C$ and $g \in G$, $Cg$ is also a perfect code.

\begin{theorem}
\label{thm:dih}
Let $\D_{2n}=\langle a,b\mid a^n=b^2=(ab)^2=e\rangle$, and let $H$ be a proper subgroup of $\D_{2n}$.
\begin{itemize}
\item[(a)] $H$ is a perfect code of $\D_{2n}$ if and only if either $H\nleq\langle a\rangle$, or $H\leq\langle a\rangle$ with at least one of $|H|$ and $n/|H|$ odd.
\item[(b)] $H$ is a total perfect code of $\D_{2n}$ if and only if either $H\nleq\langle a\rangle$, or $H\leq\langle a\rangle$ with $|H|$ even and $n/|H|$ odd.
\end{itemize}
\end{theorem}

\begin{proof}
By Lemma~\ref{lem32}, we only need to prove that every subgroup $H=\langle a^t,a^sb\rangle$ with $t > 1$ dividing $n$ and $0\leq s\leq t-1$ is a perfect code as well as a total perfect code of $\D_{2n}$. Note that
\[
\overline{H}=(e+a^t+a^{2t}+\dots+a^{n-t})(e+a^sb).
\]
Let $R=\{b,ba,ba^2,\dots,ba^{t-1}\}$ and $S=\{a^{s-1}b,a^{s-2}b,\dots,a^{s-t+1}b\}$. Then $R^{-1}=R$ and $S^{-1}=S$. Moreover,
\begin{align*}
\overline{R}\cdot\overline{H}&=b(e+a+a^2+\dots+a^{t-1})(e+a^t+a^{2t}+\dots+a^{n-t})(e+a^sb)\\
&=b(e+a+a^2+\dots+a^{n-1})(e+a^sb)\\
&=b(e+a+a^2+\dots+a^{n-1})+b(e+a+a^2+\dots+a^{n-1})a^sb\\
&=b(e+a+a^2+\dots+a^{n-1})+(e+a+a^2+\dots+a^{n-1})\\
&=\overline{\D}_{2n}.
\end{align*}
Since $\overline{S}=a^s(\overline{R}-b)$, we have
\begin{align*}
\overline{S\cup\{e\}}\cdot\overline{H}&=(a^s(\overline{R}-b)+e)\overline{H}
=a^s\overline{R}\cdot\overline{H}-a^sb\overline{H}+\overline{H}=a^s\overline{G}-\overline{H}+\overline{H}=\overline{\D}_{2n}.
\end{align*}
Hence by Lemma~\ref{thm1}, $H$ is a total perfect code in $\Cay(\D_{2n},R)$ and a perfect code in $\Cay(\D_{2n},S)$.
\end{proof}

\section{Perfect-code-preserving automorphisms}
\label{sec:pcp}

In this section we consider how we can construct more (total) perfect codes in a Cayley graph from known ones using automorphisms of the underlying group. With this in mind we introduce the following definition.

\begin{definition}\label{def1}
Let $G$ be a group and $\sigma$ an automorphism of $G$.
\begin{itemize}
\item[(i)] $\sigma$ is \emph{perfect-code-preserving} if for any Cayley graph $\Cay(G,S)$, whenever a subset $C$ of $G$ is a perfect code in $\Cay(G,S)$, $C^\sigma$ is also a perfect code in $\Cay(G,S)$.
\item[(ii)] $\sigma$ is \emph{total-perfect-code-preserving} if for any Cayley graph $\Cay(G,S)$, whenever a subset $C$ of $G$ is a total perfect code in $\Cay(G,S)$, $C^\sigma$ is also a total perfect code in $\Cay(G,S)$.
\end{itemize}
\end{definition}

Thus, (total-) perfect-code-preserving automorphisms provide a method for constructing more (total) perfect codes in a Cayley graph from known ones. This method cannot be deduced from the results in the previous section as $C$ in Definition \ref{def1} is not required to be a subgroup of $G$. Obviously, the set of (total-) perfect-code-preserving automorphisms of $G$ form a subgroup of the automorphism group of $G$.

An automorphism of a group is called a \emph{power automorphism} if it maps every element of the group to some power of it. It is clear that power automorphisms of a group $G$ are precisely the automorphisms fixing every subgroup of $G$. The power automorphisms of $G$ form a subgroup of the automorphism group of $G$. The study of inner power automorphisms dates back to Baer \cite{Baer1935}, and a comprehensive study of general power automorphisms can be found in~\cite{Cooper1968}.

\subsection{An example for abelian groups}

\begin{theorem}\label{thm4a}
Let $G$ be an abelian group and $\sigma$ a power automorphism of $G$. If $A$ and $B$ are subsets of $G$ such that $\overline{A}\cdot\overline{B}=\overline{G}$, then $\overline{A^\sigma}\cdot\overline{B}=\overline{G}$. In particular, if $(A, B)$ is a tiling of $G$, then so is $(A^\sigma, B)$.
\end{theorem}
 
This together with Lemma~\ref{thm1} implies the following result.

\begin{corollary}\label{thm4}
Let $G$ be an abelian group. Then every power automorphism of $G$ is perfect-code-preserving and total-perfect-code-preserving. In particular, for any positive integer $m$ coprime to $|G|$, the automorphism of $G$ defined by $g\mapsto g^m,\ g \in G$ is perfect-code-preserving and total-perfect-code-preserving.
\end{corollary}

To prove Theorem \ref{thm4a}, we will need the following lemma which is obtained by a direct application of the convolution property of Fourier transforms on finite groups. Given a group $G$, let $\widehat{G}$ be a complete set of inequivalent irreducible representations of $G$. For any function $f:G\rightarrow\mathbb{C}$, the \emph{Fourier transform} of $f$ is defined to be the map
\[
\widehat{f}:\quad\rho\mapsto\sum_{g\in G}f(g)\rho(g),\quad\rho\in\widehat{G}.
\]
Recall that  for any subset $A$ of $G$, $\mu_A: G \rightarrow \mathbb{C}$ as defined in \eqref{eq:mu} is the characteristic function of $A$.

\begin{lemma}\label{thm2}
Let $G$ be a group, and let $A$ and $B$ be subsets of $G$. Then the following statements are equivalent:
\begin{itemize}
\item[(a)] $\overline{A}\cdot\overline{B}=\overline{G}$;
\item[(b)] $\left(\sum_{g\in A}\rho(g)\right)\left(\sum_{g\in B}\rho(g)\right)=\sum_{g\in G}\rho(g)$ for every $\rho\in\widehat{G}$;
\item[(c)] $\widehat{\mu_A}(\rho)\widehat{\mu_B}(\rho)=\widehat{\mu_G}(\rho)$ for every $\rho\in\widehat{G}$.
\end{itemize}
\end{lemma}

\begin{proof}
Since (a) can be rewritten as
\[
\left(\sum_{g\in A}g\right)\left(\sum_{g\in B}g\right)=\sum_{g\in G}g,
\]
it is readily seen that (a) implies (b). From the definition of the Fourier transforms of $\mu_A$ and $\mu_B$ one can see that (b) implies (c).

It remains to prove that~(c) implies~(a). Define
\[
f:\quad G\rightarrow\mathbb{C},\quad g\mapsto\sum_{h\in G}\mu_A(h)\mu_B(h^{-1}g).
\]
By the convolution property of Fourier transforms, we have $\widehat{\mu_A}(\rho)\widehat{\mu_B}(\rho)=\widehat{f}(\rho)$ for every $\rho\in\widehat{G}$. Suppose that (c) holds. Then $\widehat{f}(\rho)=\widehat{\mu_G}(\rho)$ for every $\rho\in\widehat{G}$, and therefore $f=\mu_G$ by the formula of the inverse Fourier transform. This implies
\begin{align*}
\overline{A}\cdot\overline{B}&=
\left(\sum_{g\in G}\mu_A(g)g\right)\left(\sum_{g\in G}\mu_B(g)g\right)\\
&=\sum_{g\in G}g\sum_{h\in G}\mu_A(h)\mu_B(h^{-1}g)\\
&=\sum_{g\in G}f(g)g\\
& =\sum_{g\in G}\mu_G(g)g\\
& =\overline{G},
\end{align*}
completing the proof.
\end{proof}

Denote $\zeta_m = \exp(2\pi i/m)$ for any positive integer $m$. As usual, denote by $\mathbb{Q}(\zeta_m)$ the $m$th cyclotomic field ($m \ge 3$) and $\Gal(\mathbb{Q}(\zeta_m)/\mathbb{Q}) \cong \left(\mathbb{Z}/m\mathbb{Z}\right)^{\times}$ the Galois group of the Galois extension $\mathbb{Q}(\zeta_m)/\mathbb{Q}$.

\smallskip
\begin{Proof}\textit{\ref{thm4a}.}
Write $G=\langle g_1\rangle\times\dots\times\langle g_k\rangle$ with $g_j$ of order $m_j$ for $1\leq j\leq k$.
Let $\rho$ be a nontrivial irreducible representation of $G$. Then there exist integers $n_1,\dots,n_k$ such that
\[
\rho:\quad G\rightarrow\mathbb{C},\quad g_1^{a_1}\cdots g_k^{a_k}\mapsto\zeta_{m_1}^{n_1a_1}\cdots\zeta_{m_k}^{n_ka_k}.
\]
Let $m=|G|$. Clearly, $\rho(g)\in\mathbb{Q}(\zeta_m)$ for every $g\in G$. Since $\gcd(n,m)=1$, there exist $\ell\in\mathbb{Z}$ and $\tau\in\Gal(\mathbb{Q}(\zeta_m)/\mathbb{Q})$ such that $n\ell\equiv1\pmod{m}$ and $\tau(\zeta_m)=\zeta_m^\ell$.

Let $A$ and $B$ be subsets of $G$ such that $\overline{A}\cdot\overline{B}=\overline{G}$.
Since $\rho$ is nontrivial, by Lemma~\ref{thm2} we deduce from $\overline{A}\cdot\overline{B}=\overline{G}$ that
\[
\left(\sum_{g\in A}\rho(g)\right)\left(\sum_{g\in B}\rho(g)\right)=\sum_{g\in G}\rho(g)=0.
\]
This means that either $\sum_{g\in A}\rho(g)=0$ or $\sum_{g\in B}\rho(g)=0$. Note that
\begin{align*}
\sum_{g\in A}\rho(g) &= \sum_{g^\sigma\in A^\sigma}\rho(g)\\
&=\sum_{g\in A^\sigma}\rho(g^{\sigma^{-1}})\\
&=\sum_{g\in A^\sigma}\rho(g^\ell)\\
&=\sum_{g\in A^\sigma}\tau(\rho(g))\\
&=\tau\left(\sum_{g\in A^\sigma}\rho(g)\right).
\end{align*}
Since $\tau$ is a field automorphism, we see that $\sum_{g\in A}\rho(g)=0$ if and only if $\sum_{g\in A^\sigma} \rho(g) = 0$. Consequently,
\[
\left(\sum_{g\in A^\sigma}\rho(g)\right)\left(\sum_{g\in B}\rho(g)\right)=0=\sum_{g\in G}\rho(g).
\]
Next let $\rho$ be a trivial irreducible representation of $G$. Since $\overline{A}\cdot\overline{B}=\overline{G}$, we have $|A||B|=|G|$. Hence
\[
\left(\sum_{g\in A^\sigma}\rho(g)\right)\left(\sum_{g\in B}\rho(g)\right)=|A^\sigma||B|=|A||B|=|G|=\sum_{g\in G}\rho(g).
\]

Thus far we have shown that $\left(\sum_{g\in A^\sigma}\rho(g)\right)\left(\sum_{g\in B}\rho(g)\right)=\sum_{g\in G}\rho(g)$ for every irreducible representation $\rho$ of $G$. Appealing to Lemma~\ref{thm2}, we then conclude that $\overline{A^\sigma}\cdot\overline{B}=\overline{G}$.
\qed
\end{Proof}
 
\subsection{Inner automorphisms}

\begin{theorem}\label{prop3}
For any group $G$, every perfect-code-preserving inner automorphism of $G$ is a power automorphism of $G$.
\end{theorem}

\begin{proof}
Suppose that $\sigma$ is a perfect-code-preserving inner automorphism of $G$. Let $g$ be the element of $G$ that induces $\sigma$, and $H$ be a nontrivial subgroup of $G$. Suppose that there exists $h\in H$ such that $ghg^{-1}\notin H$. Then $e$ and $ghg^{-1}$ are in distinct left cosets of $H$. Take $C$ to be a right transversal of $H$ in $G$ containing $e$ and $ghg^{-1}$, and take $S=H\setminus\{e\}$. Thus by Lemma~\ref{thm1}, $C$ is a perfect code in $\Cay(G,S)$. However, $C^\sigma=g^{-1}Cg$ contains $e$ and $h$. Hence $C^\sigma$ is not a right transversal of $H$ in $G$. Thus by Lemma~\ref{thm1}, $C^\sigma$ is not a perfect code in $\Cay(G,S)$, contradicting the assumption that $\sigma$ is perfect-code-preserving. Consequently, $ghg^{-1}\in H$ for all $h\in H$. This implies that $gHg^{-1}=H$ and so $H=g^{-1}Hg$. Since this is true for every subgroup $H$ of $G$, $\sigma$ is a power automorphism of $G$.
\end{proof}

\begin{corollary}
Suppose that $G$ is a group with trivial centre. Then the only perfect-code-preserving inner automorphism of $G$ is the identity.
\end{corollary}

\begin{proof}
Denote by $Z_j$ the $j$th centre of $G$. Since the centre of $G$ is trivial by our assumption, we have $Z_1=\{e\}$ and hence
$$
Z_2=\{x\in G\mid\forall y: x^{-1}y^{-1}xy \in Z_1\}=\{x\in G\mid\forall y: x^{-1}y^{-1}xy=e\}=Z_1=\{e\}.
$$
Let $\sigma$ be a perfect-code-preserving inner automorphism of $G$ induced by $g\in G$. By Theorem \ref{prop3}, $\sigma$ is a power automorphism. Hence $g$ lies in the normalizer of every subgroup of $G$. By~\cite{Schenkman1960}, we then have $g\in Z_2$. Therefore, $g=e$ and so $\sigma$ is the identity.
\end{proof}

\noindent\textsc{Acknowledgements.}~ The authors would like to thank the anonymous referees for their helpful comments. The second author was supported by Australian Research Council grant DP150101066, and the third author by Australian Research Council grant FT110100629.

\end{document}